\newif\iffinal
\finaltrue 

\documentclass[letterpaper,11pt,reqno]{amsart}
\RequirePackage[utf8]{inputenc}
\usepackage[portrait,margin=2.25cm]{geometry}
\usepackage{style,todonotes}

\usepackage{color} 
\newcommand{\pd}[1]{{#1}}

\newcommand{\prob}{\mbox{{prob}}}

\begin{document}
\title[Multiplicative Cascade under Strong Disorder]{On Normalized Multiplicative Cascades under Strong Disorder}
\author[Dey]{Partha S.~Dey$^1$}
\address{$^1$Dept.~of Mathematics, University of Illinois at Urbana-Champaign, Altgeld Hall, 1409 W.~Green St, Urbana, IL 61801}
\author[Waymire]{Edward C.~Waymire$^2$}
\address{$^2$Dept.~of Mathematics, Oregon State University, Kidder Hall 368, Corvallis, OR 97331}
\email{psdey@illinois.edu, waymire@math.orst.edu}
\subjclass[2010]{Primary: ; Secondary: ;}

\begin{abstract}

Multiplicative cascades, under weak or strong disorder, refer to sequences of positive
 random measures $\mu_{n,\beta},  n = 1,2,\dots$, 
parameterized by a positive disorder parameter $\beta$, and
defined on the Borel $\sigma$-field ${\mathcal B}$ of $\partial T = \{0,1,\dots b-1\}^\infty$ 
for the product topology. 
The normalized cascade is defined by the corresponding sequence of
random probability measures $\prob_{n,\beta}:= Z_{n,\beta}^{-1}\mu_{n,\beta}, n = 1,2\dots,$ normalized to
a probability by the partition function $Z_{n,\beta}$.  In this note, a recent result
of Madaule~\cite[2011]{madaule} is used to explicitly construct
a family of tree indexed probability measures $\prob_{\infty,\beta}$ for 
strong disorder parameters $\beta > \beta_c$, almost surely defined on a common probability space.  
Moreover, viewing 
$\{\prob_{n,\beta}: \beta > \beta_c\}_{n=1}^\infty$ as a sequence
of probability measure valued stochastic process leads to finite dimensional weak convergence
in distribution to a probability 
measure valued process $\{\prob_{\infty,\beta}:
\beta > \beta_c\}$.   The  limit process is constructed from the tree-indexed random field of derivative 
martingales, and the Brunet-Derrida-Madaule decorated Poisson process.   A number of 
corollaries are provided to illustrate the utility of this construction.

\end{abstract}
\maketitle

\section{Introduction}\label{sec:int}

The relationship between branching random walks and multiplicative cascades  has a long history, going back to the early works of~\cite{big}
and of \cite{kahanepeyriere}, respectively.   Recent results from the latter are exploited in the present note to obtain the distribution of the normalized multiplicative cascade probability under strong disorder conditions.

{\it Branching random walks,} as discretizations of branching Brownian motion,  provide a natural probabilistic structure that 
 is known to occur, for example, in the context of  reaction-dispersion equations of the type introduced by Fisher
, Kolmogorov, Petrovskii and Piskounov; 
 see \cite{kyprianou} and references therein.    

Originating in statistical  turbulence and other areas in which singular intermittent random  distributions arise, {\it multiplicative cascades} are random measures that define prototypical models of disorder; see \cite{kahanepeyriere} for a seminal mathematical formulation  whose inspiration is attributed to Benoit Mandelbrot.  Much of the early work on multiplicative cascades involved the fine-scale (multifractal) structure of a limiting cascade distribution under conditions that have come to be referred to as \emph{weak disorder}.  In such cases the total mass defines a  positive martingale sequence with a non-trivial a.s.~limit.   In particular, the cascade measure can easily be normalized to a (random) probability measure to obtain an a.s.~weak limit.  On the other hand, while compactness of the tree boundary ensures tightness, such almost sure weak limits are not expected to exist under \emph{strong disorder}.  

However,  as shown in~\cite{john_way} and in~\cite{barralrhodesvargas}, 
respectively, there is a weak limit in probability at  \emph{critical strong disorder}, or the so-called  \emph{boundary} case, and a weak limit in distribution under strict (non-critical) strong disorder.  In particular a (random) probability can be defined in the infinite path limit under strong disorder.  
This latter result will also follow from the analysis presented here, but the focus of this note is rather on the structure of these weak limits and their mutual relation as a stochastic process indexed by the disorder parameter $\beta >\beta_c$.  Toward this goal an integral representation is provided together with a
 limiting process, in the sense of weak convergence finite
dimensional distributions,
defined almost surely as a function of $\beta$ on a single
probability space.  This is then used to provide mutual absolute continuity between the disorder limits, formulae for the Radon-Nikodym derivatives,  and  an explicit description of the genealogy near the root as  corollaries.  Moreover, it is shown that as
a probability measure valued process, the limit process indexed
by $\beta > \beta_c$
 has a.s.~continuous paths in the weak-* topology; in fact in the total-variation norm.
The basic approach is to construct a tree-indexed derivative martingale random field, and then exploit recent consequences of  superposability due to~\cite{madaule,brunetderrida}.
 
\section{Background Definitions and Notation}\label{sec:back}
To clearly describe the focus of this note
it is convenient to introduce some standard notation defining a multiplicative cascade, and its normalization to a  probability.  While the results may be more generally formulated for cascades on more general classes of trees, including Galton-Watson supercritical branching processes subject to a Kesten-Stigum condition on the offspring distribution,
we restrict the presentation to directed binary trees for simplicity of exposition.

Consider the infinite binary tree defined by the following set of vertices $
\vT=\bigcup _{n=0}^\infty \{-1,+1\}^n$ with edges defined by pairs of vertices of the form $v = (v_1,\dots,v_n),$ and its parent $v|(n-1) = (v_1,\dots,v_{n-1})$, and rooted at $\emptyset$ in correspondence with $ \{-1,1\}^0$.  The \emph{boundary} of $\vT$ is defined by 
$\partial \vT=\{-1,1\}^\dN $, with the product topology. 
An \emph{$\infty$-tree path}  is denoted by
$s=(s_1,s_2,\ldots)\in \partial T$.  We will also consider \emph{finite tree paths} 
$s=(s_1,\ldots,s_n)\in \vT\backslash \{\emptyset\}$ of length $|s| = n$, and for $s=(s_1,s_2,\ldots)\in\partial \vT$, continue to use the notation
$s|n:=(s_1,s_2,\ldots,s_n)$, read ``$s$ restricted to $n$'', for  truncation.  Also, for $v\in \vT, k=|v|\le n$ we define 
\[
\gD(v):=\{s\in \partial\vT: s|k=v\} \text{ and } \gD_n(v):=\{s\in \{-1,+1\}^n : s|k=v\}
\]
as the $\infty$-paths passing through the vertex $v$ and the vertices at level $n$ below the vertex $v$, respectively. 

Suppose one is given  a collection $\{X_v\mid v\in \vT\}$ of i.i.d.~positive random variables indexed by $\vT$ and defined on a probability space $(\gO,\cF,\pr)$.  Denote by $X$ a generic random variable having the common distribution of each $X_v$ and assume that $\E(X)=1$. 
Let
$\gl(ds)=\left( \frac{1}{2}\delta_{+1}(ds) + \frac{1}{2}\delta_{-1}(ds) \right)^\dN$
on $(\partial \vT, \cB)$, and define the sequence of positive (random) measures $\mu_n(ds), n\ge 1$, 
absolutely continuous with respect to $\gl(ds)$, via
their sequence of Radon-Nikodym derivatives given by
\begin{align}\label{RNDcascade}
	\frac{d{\mu}_n}{d\gl}(s)= \prod_{j=1}^n X_{s|j}. 
	\end{align}
Note that $\int_{\partial \vT}f(s)\mu_n(ds), n\ge 1,$ is a bounded martingale for any bounded, continuous
function $f$ on $\partial \vT$.  The corresponding sequence of normalized (random) probability measures $\text{prob}_n(ds)$ on $\partial T$ is defined by	
\begin{align}\label{RNDprob}
	\frac{d\text{prob}_n}{d\gl}(s)=M_n^{-1}\prod_{j=1}^n X_{s|j},  
\end{align}
where  the  \emph{partition function}  $M_n$ normalizes $\mu_n$, to a probability measure. Note that
\begin{align}\label{partition}
\begin{split}
	M_n&=  2^{-n}\sum_{|s|=n} \prod_{j=1}^n X_{s|j}
\end{split}                                              
\end{align}
has mean $1$. 
The sequence of non-normalized measures $\mu_n(ds), n\ge 1,$ is referred to as a \emph{multiplicative cascade} and is the main object of our analysis.

In this framework, the notions of \emph{weak disorder} and \emph{strong disorder}, e.g., see  Bolthausen~\cite{bolthausen91} for these notions in the  present context, 
provide a well-known dichotomy defined in terms of the asymptotic behavior of the partition function as follows.  First note that the sequence of (normalized) partition functions  $M_n, n\ge 1,$  is a positive martingale, so  $M_\infty:=\lim_{n\rightarrow\infty}M_n$ exists a.s.~in $(\gO,\cF,\pr)$.  
By positivity of the factors defining the path probabilities, the event $[M_\infty=0]$ is a tail event and thus by Kolmogorov's zero-one law, $\pr({M}_\infty=0)$ must equal zero or one.  Kahane and Peyri\`ere~\cite{kahanepeyriere} for multiplicative cascades, and (independently) Biggins, Hammersley and Kingman~\cite{big}, 
for branching random walks, had already obtained the following precise conditions for the disorder dichotomy:
\begin{align}\label{eq:dicho}
	\pr({M}_\infty>0) = 1 \quad & \Longleftrightarrow \quad \E(X\log X) < \log2.
\end{align}
In the case for which $[M_\infty>0]$ a.s., the cascade is said to be in a state of \emph{weak disorder}, whereas if $[M_\infty=0]$ a.s., the cascade is in a state of \emph{strong disorder}. Note that the deterministic environment $X\equiv 1$ a.s.~can be regarded informally as the ``weakest'' of the weak disorder regimes where $M_{n}\equiv 1$ and $\mu_{n}(ds)\equiv \gl(ds)$.
The special case 
\begin{align}\label{boundary}
	\E(X\log X)= \log 2, 
\end{align}
belongs to the strong disorder regime as  {\it critical disorder}, or the {\it boundary case}.  For example, in the case when $X=\exp(-\gb N -\gb^2/2)$ with $N$ being standard normal distributed, the boundary case corresponds to $\gb=\sqrt{2\log 2}$, with the strong disorder regime obtained for $\gb\ge \sqrt{2\log 2}$. 

To describe the limit distribution of  the (re-scaled) partition function in the critical case $\E(X\log X)=\log2$ or $\E(X(\log2-\log X))=0$, let us recall the 
{\it derivative martingale}  in the boundary case of the branching random walk;
see \cite{bigkyp04}. We have that
\begin{align}
D_n = 2^{-n}\sum_{|s|=n}  \sum_{j=1}^n\bigl(\log 2 -  \log X_{s|j}\bigr)\prod_{j=1}^n X_{s|j},  \ n\ge 1, 
\end{align}
is an $L^1$-bounded martingale having an a.s.~positive limit $D_\infty$, and referred
to as the {\it derivative martingale}; see  \cite{kyprianou} for additional historic background in the contexts of
branching random walk and branching Brownian motion.  Under some natural regularity conditions on $X$, Aidekon and Shi~\cite{aidshi} proved that $\sqrt{n}M_n/D_n$ converges in probability to $\sqrt{2/\pi\gs^2}$ where 
$\gs^2:=\E(X(\log2-\log X)^2)$. 

Additional insight into the relevance of the derivative martingale  to multiplicative cascade theory can be obtained by considering the basic stochastic cascade recursion 
	\begin{align}\label{randomrecur}                     
		B \equald A_{-1}B_{-1} + A_{+1}B_{+1}, 
	\end{align}
	where $B_{\pm 1}$ are i.i.d.~non-negative r.v.s having the same distribution as $B$, and $A_{\pm 1}$ are i.i.d.~non-negative r.v.s having mean $1/2$, and independent of $B_{\pm 1}$.  The recursion~\eqref{randomrecur} is a well-studied recursion in a variety of contexts, see \cite{bigkyp05}
	
	Under weak disorder $B = M_\infty$ is the nontrivial solution for $A_{\pm1}=\frac{1}{2}X_{\pm 1}$, unique up to positive constant multiples.  However, at strong disorder one has $M_{\infty} = 0$ a.s; \ie a trivial solution to~\eqref{randomrecur}.   Nonetheless there is a  nontrivial solution in the (non-lattice) boundary case,  namely a constant multiple of $D_\infty$; see~\cites{kyprianou,hu_shi}.  
	
	It is generally well-known as a result of early work originating in \cite{durrettliggett83} 
	 that under strong disorder the solution (fixed point) of the random recursion (\ref{randomrecur})  coincides with a multiple of a  L\'evy stable process stopped at $D_\infty$;  see~ \cite{kyprianou} for a summary and extensions.   The results to follow provide a more detailed analysis of 
the structure of this solution, through its explicit connections to the extremes of the associated branching random walk, that facilitates the almost sure construction of the limit probabilities 
$prob_\infty(ds)$. 

To close this section let us note that \emph{tree polymers} provide an essentially equivalent formulation that can
be described as follows. Namely,
when $X=\exp(-\gb W)/\E e^{-\gb W}$ for some $\gb>0$, the sequence of random probability measures $\{\text{prob}_n(ds): n\ge 1\}$ is also referred to as a \emph{tree polymer} on $(\partial \vT, \mathcal{B})$ at inverse temperature $\gb$.  
Assuming that $W$ is a random variable with $\varphi(\gb):=\E e^{-\gb W}
<\infty$ for all $\gb\ge 0$, the dichotomy~\eqref{eq:dicho} for the r.v.~$X=\varphi(\gb)^{-1}e^{-\gb W}$ gives the critical disorder as $\gb=\gb_{c}$ where $(\gb^{-1}\log(2\varphi(\gb)))'\bigr|_{\gb=\gb_c}=0$ and the weak disorder as $\gb<\gb_c$.  By centering and scaling appropriately, \ie working with $\gb_c W+\log(2\varphi(\gb_c))$ instead of $W$, without loss of generality we can assume the so-called \emph{boundary case} defined by
\begin{align}\label{eq:ct}
	\E(e^{-W})=\frac12 \text{ and } \E(We^{-W})=0.
\end{align}
Thus with $X_{v}=\varphi(\gb)^{-1}e^{-\gb W_v}, v\in\vT$ the strong disorder corresponds to $\gb\ge 1$. 
The exponent  defined by  \eqref{eq:bndryalpha} is given by $\alpha = 1/\beta$ in this framework.

We define the energy of a finite path $s\in\vT$ as
\[
	H(s) =\sum_{i=1}^{|s|}W_{s\mid i} \text{ for } s\in\vT
\]
and will sometimes use $H_{n}(s)$ instead of $H(s)$ when $|s|=n$ to emphasize the dependence on $n$. Also, in this context the \emph{partition function} is defined as 
$Z_{n}(\gb) :=\sum_{|s|=n} e^{-\gb H(s)}$. We also define, for $v\in\vT,|v|\le n$
\begin{align}\label{eq:vertexpf}
Z_{n}(\gb;v)= \sum_{s\in \gD_n(v)} e^{-\gb (H(s)-H(v))}.
\end{align}
Then \eqref{eq:vertexpf}  can be understood as the partition function at the vertex $v$. Clearly $Z_{n}(\gb)=Z_{n}(\gb;\emptyset)$. 
One may note that the scaling of the partition function implies a certain centering of the branching random walkers induced by the path energies $H_n(s), |s| = n,$ that may explicitly be expressed as follows:
\begin{align*}
	n^{{\frac32}\gb} Z_n(\gb)           
	= \sum_{|s|=n}e^{-\gb(H_n(s) -{\frac32}\log n)}.                                          
\end{align*}

When $X=\varphi(\gb)^{-1}e^{-\gb W}$, we will use $\mu_{n,\gb}$ and $\text{prob}_{n,\gb}$ for~\eqref{RNDcascade} and \eqref{RNDprob}, respectively. Note that  the normalization constant $M_n$ in \eqref{partition}, is the same as $(2\varphi(\gb))^{-n}Z_n(\gb)$. Also, we have for $v\in\vT, |v|<n$
\begin{align*}
\mu_{n,\gb}(\gD(v)) &= (2\varphi(\gb))^{-n}e^{-\gb H(v)} Z_{n}(\gb;v)\\
 \text{ and }\text{prob}_{n,\gb}(\gD(v)) &= e^{-\gb H(v)} Z_{n}(\gb;v)/Z_{n}(\gb).
\end{align*}

We will also use the following definitions for finite-dimensional convergence. 

\begin{defn}\label{def21}
 Let $\partial T = \{0,1\}^\infty$ with the product topology.  
Let $C_{\mbox{\tiny fin}}$ denote the set of bounded, (continuous) functions $g:\partial T\to R$ depending on 
finitely many coordinates. Let ${\cF}_n, n = 1,2,\dots,$
be the filtration generated, respectively, by the coordinate projections $\pi_j, j\le n$, where 
$\pi_j(t) = t_j$. Suppose that $\nu_n$ is a sequence of probabilities on $(\partial T, {\cF}_n)$,
and ${\cB} = {\cF}_\infty$. Then we say $\nu_n$ converges in finite-dimensional distribution 
to $\nu$ if $\int_{\partial T}g(t)\nu_n(dt)\to \int_{\partial T}g(t)\nu(dt)$ for all $g\in C_{\mbox{\tiny fin}}$.
\end{defn}

\begin{defn}\label{def22}
Suppose $\Pi_n = \{\nu_n(\lambda,dt): \lambda\in I\}$ is a sequence of 
probability measure valued stochastic processes on $({\cF}_n\times\Omega,\lambda\times P)$,
respectively,
where $\lambda$ is Haar measure on ${\cB}\supset {\cF}_n$ for all $n$.
Then we say one has finite dimensional weak convergence in distribution
 of $\Pi_n$ to $\Pi$ if for 
any finite $\lambda_1,\lambda_2,\dots,\lambda_m$ in $I$, one has 
$(\int_{\partial T}g_i(t)\Pi_n(\lambda_i,dt))_{i=1}^m\to (\int_{\partial T}g_i(t)\Pi(\lambda_i,dt))_{i=1}^m$
 in distribution for all $g_i\in C_f, 1\le i\le m.$
 \end{defn}

\section{Main Results}\label{sec:main}
With the previous section as background, let $X$ be a positive random variable with mean one and satisfying the \emph{strict strong disorder} condition $\E(X\log X)>\log 2$ for the multiplicative cascade defined in~\eqref{RNDcascade}.  By calculations of the
type given in \cite{holleyway}, it is easy to see the following (scale invariant) fact.
\begin{lem}
Assume that ${\mathbb E}X = 1$ and ${\mathbb E}X\log X > \log 2.$  
Then there is a unique  $\ga\in(0,1)$ such that
\begin{equation}\label{eq:bndryalpha}
\E\left(\frac{X^{\ga}}{\E X^{\ga}}\log \frac{X^{\ga}}{\E X^{\ga}}  \right) = \log 2.
\end{equation}
\end{lem}
\begin{proof}
Let $\rho(\alpha) = \mathbb{E}(\frac{X^\alpha}{\E X^\alpha}\log_2X)$.
The assertion is equivalent to the existence of a unique $\alpha\in (0,1)$ such that 
$$\alpha\rho(\alpha) - \log_2{\mathbb E}X^\alpha = 1.$$
The left side is zero at $\alpha = 0$ and, at $\alpha = 1$ it is $\rho(1) =  {\mathbb E}X\log_2 X > 1$.
Moreover, it follows from the Cauchy-Schwarz inequality that the left side is also an increasing 
function of $\alpha$.  So the assertion
 follows from these observations together with continuity  of the left hand side.
\end{proof}

Let us define
\begin{align}\label{eq:XtoW}
W:= \log 2 + \log\E(X^{\ga})-\ga\log X
\end{align}
so that $W$ satisfies $\E(e^{-W})=1/2$ and $\E(We^{-W})=0$. 




Next we construct a collection of random variables indexed by the vertices of the infinite binary tree that will appear in the joint convergence of the partition functions at different vertices  at the critical disorder, \ie $\gb=1$. This 
\emph{tree-indexed derivative martingale random
field} provides an essential ingredient of the eventual  construction.
Recall that, the derivative martingale is defined as
\[
D_n:= \sum_{|s|=n} H(s)e^{-H(s)}	
\]
and has an a.s.~positive limit $D_{\infty}$ satisfying the distributional recursion~\eqref{randomrecur}, \ie $e^{-W_{-1}}D_\infty(-1) + e^{-W_{+1}}D_\infty(+1) \equald D_{\infty}$ where $D_{\infty}(\pm1)$ are i.i.d.~$\sim D_{\infty}$.

Assume that $\{W_{v}:v\in\vT\}$ is a collection of i.i.d.~random variables each distributed as $W$ and indexed by $\vT$. Fix a positive integer $k$. For $v\in\vT, |v|=k$, let $D(v)$ be i.i.d.~copies of $D_{\infty}$. Now inductively for $i=k-1,k-2,\ldots,0$, define 
\begin{align}\label{eq:ind}
D(v):=D(v,-1)e^{-W_{v,-1}} + D(v,+1)e^{-W_{v,+1}} \text{ for } v\in\vT, |v|=i.
\end{align}
It is easy to see that for any fixed $i\le k$, $\{D(v),|v|=i\}$ are i.i.d.~copies of $D_{\infty}$ and thus $\{D(v): |v|\le k\}$ is a consistent family of distributions. By Kolmogorov's consistency theorem there exists a
(denumerable) tree-indexed collection of random vectors
\begin{align}
\vD_{\infty}:=\{(W_{v},D_{\infty}(v)): v\in\vT\}
\end{align}
such that the finite-dimensional distribution restricted to $\{v:|v|\le k\}$ is given by the above construction~\eqref{eq:ind}.

Now define the interval $I(\emptyset)=[0,D_{\infty}(\emptyset))$. One 
can think of the tree-indexed derivative martingales
$\vD_{\infty}$ as providing a way in which to partition the interval  $I(\emptyset)$ into successively smaller intervals. Define 
\begin{align*}
I(-1) &=[0,e^{-W_{-1}}D_{\infty}(-1))\\
\text{ and } I(+1) &=[e^{-W_{-1}}D_{\infty}(-1),e^{-W_{+1}}D_{\infty}(+1)+e^{-W_{-1}}D_{\infty}(-1)).
\end{align*} 
Note that $D_{\infty}(\emptyset)=e^{-W_{-1}}D_{\infty}(-1)+e^{-W_{+1}}D_{\infty}(+1)$  a.s.~by construction and thus $I(+1), I(-1)$ is a partition of $I(\emptyset)$. Now to define $I(v)$ for $v\in\vT, |v|=k$, consider the lexicographic ordering on $\{-1,+1\}^k$, \ie  for $u,v\in\{-1,+1\}^k$, $u\prec v$ iff  there exists $i\in\{0,1,\ldots,k\}$ such that $u|i=v|i$ and $u_{i+1}<v_{i+1}$. Now, for $v\in\vT, |v|=k$ define
\begin{align}
I(v) := \biggl[\ \sum_{u\prec v} e^{-H(u)}D_{\infty}(u) , e^{-H(v)}D_{\infty}(v)+ \sum_{u\prec v} e^{-H(u)}D_{\infty}(u)\biggr).
\end{align}
One can easily check that the collection of intervals $\{I(v): v\in\vT\}$ respects the tree structure in terms of set-inclusion, i.e., if $v$ is an ancestor of  $u$ then $I(u)\subseteq I(v).$ 

Here we note that, any infinite path $s\in\partial \vT$ can be represented by a point $t(s)\in I(\emptyset)$ and conversely any point $t_{0}\in I(\emptyset)$ corresponds to a unique path $s=s(t_0)\in\partial \vT$ in the sense that $\{t\}=\bigcap_{k=1}^{\infty} {I}(s|k)$.
	

Let $\theta>0$ be a fixed real number. Consider a decorated (or marked) Poisson point process $\cN$ in $\dR\times [0,\infty)$ with intensity measure $e^x dtdx, (x,t)\in \dR\times [0,\infty)$ and the decoration at the point $(x,t)$  given by $V_{x,t}$ which are i.i.d.~copies of a point process $V$. 
Let $\{(W_v,D_{\infty}(v)): v\in\vT\}$ be a collection of random variables indexed by the vertices of $\vT$ as constructed above,
and  independent of $\cN$. Fix a real number $\theta>0$. 

Now for any $\ga\in (0,1)$ and $v\in \vT$, define
\begin{align}
\cI_{\ga}(v)= 
\int_{\dR\times\theta I(v)}e^{-{z/\alpha}}\cN(dz\times dt) =
\sum_{(x,t)\in\cN}  e^{- x/\ga}\ind_{t/\theta \in I(v)} \sum_{y\in V_{x,t}}e^{-y/\ga}.\label{eq:cI}
\end{align}

With these preliminaries, the
 main result of this note may now be stated as follows. 

\begin{thm} 
	\label{mainthm}
	Assume that the distribution of $W$ is non-lattice, satisfies the boundary condition~\eqref{eq:ct} and the following size-biased moment is finite: 
	\begin{align}
		\E (W^2+\log_+(e^{-W} + We^{-W}))e^{-W} <\infty
	\end{align}
	Then, we have for any $\gb_{1},\gb_{2},\ldots,\gb_{k}\in (1,\infty)$ and $v_{1},v_{2},\ldots,v_{k}\in\vT$
\begin{align*}
\{ n^{3\gb_i/2}e^{-\gb H(v_i)}Z_n(\gb_i;v_i): i=1,2,\ldots,k\} \Rightarrow \{ \cI_{1/\gb_i}(v_i): i=1,2,\ldots,k\} 
\end{align*}
in the sense of convergence in distribution for some $\theta>0$ and some point process $V$. In particular, there are random probability measures $\text{prob}_{\infty,\gb}(ds)$ on $\partial \vT$ parameterized by $\gb > 1$ and defined on a common probability space such that 
$$
\{\text{prob}_{n,\gb}(\Delta(v)) : v\in\vT\} \Rightarrow\{\text{prob}_{\infty,\gb}(\Delta(v)) : v\in\vT\},
$$
where $\Rightarrow$ denotes finite-dimensional convergence in distribution and $$\text{prob}_{\infty,\gb}(\Delta(v)):=\cI_{1/\gb}(v)/\cI_{1/\gb}(\emptyset) 
\equiv \cI^{-1}_{1/\gb}(\emptyset)\int_{\dR\times\theta I(v)}e^{-{\beta z}}\cN(dz\times dt)
, \quad v\in\vT. 
$$
\end{thm}

\begin{rem}
The physics of random distributions of the type obtained here can be 
phrased in terms of  \emph{metastates} as defined in~\cite{aizwehr}. In fact, we prove finite-dimensional  convergence of the \emph{joint distribution} of the
$\text{prob}_{n,\beta}$'s and the disorder, i.e., the $W_v$'s.
This defines a metastate in the Aizenman-Wehr
sense~\cite{aizwehr}. 
Related notions occur in the mathematical physics
 literature~\cites{newman_stein, bovier}. For example,
a metastate in the sense of Newman-Stein requires that one
condition on the disorder first, and then obtain the limit of an empirical distribution of the $\prob_{n_k}$'s along some sparse (but deterministic)
 subsequences $n_k$.
 The more purely probabilistic content follows the perspective of Aldous'
\cite{aldous} \emph{objective approach} in which one may view the 
construction of the random objects $\prob_{\infty,\gb}$ as natural
stochastic structures associated with the sequence $\prob_{n,\gb},
n\ge 1$ via a weak convergence in distribution; e.g. see 
Corollary \ref{gencor} below.
\end{rem}

\pd{
Here we mention that in the strong disorder regime, \ie $\gb>1$,  the measures $\mu_{n,\beta}$ do not have a non-trivial limit. However, the $\gs$-finite measure $n^{-3\gb/2}\cdot\mu_{n,\gb}$ has the weak limit $\mu_{\infty,\gb}:=\cI_{1/\gb}(\emptyset)\ \prob_{\infty,\gb}$ over the collection of sets $\gD(v),v\in \vT$ and $\mu_{\infty,\gb}(\gD(v))$ can be written as a scale mixture of $1/\gb$-stable random variables. 
}

As a consequence of the explicit construction 
we can see that 
 the limiting measures $(\prob_{\infty,\gb}, \gb>1)$ are defined on the same probability space and are mutually absolutely continuous on $\partial T$ $\Omega$-a.s.  
By the definition of the intervals $(I(v))_{v\in \vT}$, any infinite path $s\in \{-1,+1\}^\infty$ in the binary tree will be represented by a point $t(s)$ in the interval $I(\emptyset)$.  More specifically, with this notation, one has the following immediate consequence.

\begin{cor}
$(\text{prob}_{\infty,\gb}, \gb>1)$ are defined on the same probability space and are mutually absolutely continuous  with  the Radon-Nikodym derivative of $\text{prob}_{\infty,\gb_1}$ with respect to ~$\text{prob}_{\infty,\gb_2}$ at the infinite path $s$ (with corresponding time 
point $t(s)$) given by
\[
\frac{d\prob_{\infty,\gb_1}}{d\prob_{\infty,\gb_2}}(s) =
\frac{C_{\beta_1}(t(s))  \cI_{1/\beta_2}(\emptyset)} {C_{\beta_2}(t(s)) \cI_{1/\beta_1}(\emptyset)}.
\]
where  the $\beta$-contribution for a single point $t_0$ is given by
\[
C_\beta(t_0) := \sum_{(x,t)\in\cN: t=t_0} e^{-\beta x} \sum_{y\in V_{x,t}} e^{-\beta y}.
\]
\pd{Moreover, the sample paths of the (probability) measure-valued process $\gb\mapsto \prob_{\infty,\beta}$ are a.s. continuous for the total variation norm and, hence, weak-* topology, \ie as $\gb_{n}\to \gb>1$ one has $\prob_{\infty,\gb_{n}}$ converges weakly to $\prob_{\infty,\gb}$. }
\end{cor}
\begin{proof}
Observe that  the Poisson process is independent of the intervals.
 The $\beta$-contribution for a single point $t_0$
 is nonzero for countably infinitely many $t$'s and the support set for $t_{0}$, projection of $\cN$ on the second co-ordinate, is independent of $\beta$. Continuity of the process $\gb\mapsto \prob_{\infty,\beta}$ in the total variation norm follows from the absolute continuity using Scheffe's theorem, and continuity of the respective Laplace transforms appearing in
the Radon-Nikodym derivatives. This implies the asserted 
continuity in the weak-* topology.
\end{proof}
\pd{
In fact, the random mapping $\beta\mapsto\prob_{n,\beta}$, can be considered as a random process defined on $(1,\infty)$ and taking values in $\mathcal{M}(\{0,1\}^\infty)$, the space of probability measures on $\{0,1\}^\infty$ endowed with the weak* topology. 
The main result of this paper states that this sequence of processes converges weakly to a limiting process $\beta\mapsto\prob_{\infty,\beta}$, in the finite-dimensional sense (see Definition~\ref{def21} and \ref{def22}). }


 It may be noted that similar formulae are known for other models of disorder, such as the \emph{random energy model} (REM), and \emph{generalized random energy model} (GREM), introduced by~\cites{derrida85, ruelle87, neveu92} and related by mean-field type formulations.  It was shown by \cite{bertoinlegall00} as a consequence of~\cite{neveu92} that the genealogy of the GREM is given by the Bolthausen-Sznitman coalescent.  It is interesting to note the manner in which the asymptotic results for the multiplicative cascade model \emph{differ} from those of GREM, yet
 remain within the general framework of  $\Lambda$-coalescence (for non-uniform
$\Lambda$.)  This is elaborated upon
with related comments are included at the close of this note. 

Another specific  by-product of Theorem~\ref{mainthm} is that
one can easily  find the limiting distribution of the genealogical tree of randomly chosen $k$ vertices in $\{-1,+1\}^n$ from the distribution $\text{prob}_{n,\gb}$. Recall that for $v=(v_1,v_2,\ldots,v_n)\in\vT$ and an integer $k\le n$, we have $v\mid k=(v_1,v_2,\ldots,v_k)$. 
Consider the decorated Poisson process as given in equation~\ref{eq:cI}. Let $\nu$ be the (random) probability measure supported on $I(\emptyset)$ so that 
\[
\nu'_\gb(t_0) = \frac{1}{ \cI_{1/\gb}(\emptyset) }\sum_{(x,t)\in\cN: t=t_0} e^{-\beta x} \sum_{y\in V_{x,t}} e^{-\beta y}.
\]
Let $\nu_{\gb}$ be the probability measure on $\partial\vT$ such that $t(s)\sim \nu'_\gb$ when $s\sim\nu_{\gb}$. 
The following corollary follows easily from Theorem~\ref{mainthm}. 

\begin{cor}
\label{gencor}
Let $\vv_{1},\vv_{2},\ldots,\vv_{k}$ be $k$ many i.i.d.~vertices from the probability measure $\text{prob}_{n,\gb}$ on $\{-1,+1\}^{n}$. Let $\vu_{1},\vu_{2},\ldots,\vu_{k}$ be $k$ many i.i.d.~vertices from the probability measure $\nu_{\gb}$. Then for any fixed integer $k$ we have
\[
(\vv_{1}\mid k,\vv_{2}\mid k,\ldots,\vv_{k}\mid k)\weakc (\vu_{1}\mid k,\vu_{2}\mid k,\ldots,\vu_{k}\mid k)
\]
as $n\to\infty$.
\end{cor}
This implies local convergence of the genealogical tree for randomly chosen $k$ vertices from $\text{prob}_{n,\gb}$ near the root.

Finally let us record that a companion
formulation of weak convergence in distribution can be given in terms of 
Fourier transforms as follows.
\begin{cor}
At any strong disorder $\gb>1$, for each finite set $F\subseteq\mathbb{N}$
\begin{equation*}
\widehat{\text{\emph{prob}}}_{n,\gb}(F) \Rightarrow \widehat{\text{\emph{prob}}}_{\infty,\gb}(F)  \text{ in distribution},
\end{equation*}
where $\widehat{\text{\emph{prob}}}_{n,\gb}, n\ge 1,
\widehat{\text{\emph{prob}}}_{\infty,\gb}$ denote their respective Fourier
transforms as 
probabilities on the compact abelian multiplicative
group $\partial T$ for the product topology.
\end{cor}

\begin{proof}
The continuous characters of the group $\partial T$ are given by
$\chi_F(t) = \prod_{j\in F}t_j $ for finite sets $F\subseteq\mathbb{N}$. In particular there are only countably many characters of $\partial T$.  From standard Fourier analysis it follows that we need only show that
\begin{align*}
\lim_{n\rightarrow\infty}\E_{\text{prob}_n}\chi_F 
= \E_{\text{prob}_\infty}\chi_F\text{ in distribution}
\end{align*}
for each finite set $F\subseteq\mathbb{N}$.  Let $m=\text{max}\{k:k\in F\}$.  Then for $n>m$,
\begin{align*}
\E_{\text{prob}_{n,\gb}}\chi_F 
&= \int_{\partial \vT}\chi_F(s)\frac{d\text{prob}_n}{d\gl}(s)\gl(ds)\\
&= \sum_{|v|=m} \prod_{j\in F}v_j\cdot Z_n(\gb;\emptyset)^{-1} e^{-\gb H(v)} Z_n(\gb;v)\\
&= \sum_{|v|=m} \prod_{j\in F}v_j\cdot e^{-\gb H(v)} \cdot \frac{n^{3\gb/2}Z_n(\gb;v)}{n^{3\gb/2}Z_n(\gb;\emptyset)}
\Rightarrow \E_{\text{prob}_{\infty,\gb}}\chi_F
\end{align*}
where the convergence is in distribution.
\end{proof}

\section{Proof of Main Result}\label{sec:proof}
In recent years there has been a rapidly growing literature on the asymptotics of the extremes of branching random walks.  Relatively long, technical papers have provided a  refined understanding of the behavior of the right (or left) most particles in branching random walks;  e.g., ~\cites{aidekon,berestyckietal,hu_shi,madaule,brunetderrida}. 
This theory will be exploited to provide a coupled relation
between  the asymptotic distributions  of the  partition functions, suitably scaled,  for a general class of multiplicative cascades under strong disorder and non-lattice energy distributions, as a function of the disorder parameter.  In particular, two essential structures underlying the results here are:

(a) Biggins-Kyprianou's version of the \emph{derivative martingale};
see \cite{brunetderrida} and \cite{bigkyp04}, respectively, where these ideas arise in connection with the extremes of branching random walks, and
   
(b) Brunet-Derrida's notion of \emph{superposability}.

The role of the derivative martingale was previously explained above.
As noted, the construction of the tree-indexed derivative field
 is an essential element of the a.s. construction of the weak limits in distribution of the normalized cascade probabilities.  Another is that of 
\emph{superposability}
of extremal point processes introduced
by \cite{brunetderrida}, 
together with the (conjectured) corresponding representation as a 
decorated Poisson (cluster)  process,  rigorously
established by \cite{madaule}. 
		
	

 Specifically,
\begin{definition}
	A point process $N$ on $\dR $ is said to be superposable if, for an independent copy $N^\prime$ and any $a,b\in \dR$ such that $e^{-a} + e^{-b} = 1$,
	\begin{align*}
		T_aN + T_bN^\prime \equald N, 
	\end{align*}
	where $T_x\bigl(\sum_y\delta_y\bigr) = \sum_y\delta_{y + x},\ x\in\dR.$
\end{definition}

The basic example of a superposable point process is the Poisson process on $\dR $  with intensity $e^xdx$.  This is the well-known point process of extremes of a centered and scaled i.i.d.~Gaussian sequence.
More generally, a superposable point process is infinitely divisible and, therefore, it follows that it must be a Poisson cluster point process.  Based on analogous results for branching Brownian motion,  it had been conjectured in~\cite{brunetderrida} that the only superposable point processes were Poisson cluster processes with Poisson intensity $\theta e^xdx, \theta>0$. This was recently proven as a consequence of infinitely divisibility, and also as a consequence of LePage representation theory,
see~\cite{abk13,maillard13}.  In the context of the present note  it may also be interesting to note that Poisson cluster processes must be \emph{associated} in the sense of positive dependence (or FKG inequalities);~\cites{burtonway86,evans}.

Another conjecture by \cite{brunetderrida} was recently proven in \cite{madaule} extending the above quoted result for i.i.d.~Gaussian exremes to the extremes of the energies $H_n(s), |s|=n,$ centered and scaled.  In particular,  it is shown that in the boundary case the point process of extremes  is superposable.  More specifically, in the notation of the present article,

\begin{thm}[Theorem~1.1 in~\cite{madaule}] \label{thm:madaule} 
	Assume that the distribution of $W$ satisfies the condition of Theorem~\ref{mainthm}. Let $N_n = \sum_{|s| = n}\delta_{H(s) - {\frac32}\log n + \log D_\infty}$. 
Then $(N_n,D_n)$ converge jointly in distribution to $(N_\infty, D_\infty)$ where $N_\infty = \sum_{k\ge 1}\sum_{y\in V_k}\delta_{x_k + y}$ is a Poisson cluster  point process on $\dR$ with Poisson center process $\{x_k: k\ge 1\}$ having intensity $\theta e^{x}dx, x\in \dR $ for some $\theta >0$, $V_k$'s are i.i.d.~copies of  some point process $V$ and $D_{\infty}$ is independent of $N_{\infty}$.
\end{thm}

An easy consequence of Theorem~\ref{thm:madaule} (Theorem 2.4 in~\cite{madaule}) and the fact that
\[
n^{{\frac32}\gb}{Z}_n(\gb) = D_\infty^\gb \sum_{|s|=n}e^{-\gb(H(s) - {\frac32}\log n+\log D_\infty)}
= D_\infty^\beta\int_\dR e^{-\beta z}N_n(dz),
\]
 is that for any fixed $\gb_1,\gb_{2},\ldots,\gb_{k}\in(1,\infty)$ we have
\begin{align}
\bigl(n^{3\gb_i/2}Z_n(\gb_i),& i=1,2,\ldots,k\bigr) 
\Longrightarrow 
\bigl( D_{\infty}^{\gb_i}\sum_{k\ge1}e^{-\gb_ix_k}\sum_{y\in V_k}e^{-\gb_i y}, i=1,2,\ldots,k\bigr)
\end{align}
where $(x_k,k\ge 1)$ are points of a Poisson point process with intensity $\theta e^xdx, x\in\dR$, $V_{k}$'s are i.i.d.~copies of some point process $V$ and $D_{\infty}$ is the limiting derivative martingale independent of everything else.

Now let $\cN$ be a Poisson point process in $\dR\times [0,\infty)$ with intensity measure $e^x dtdx, (x,t)\in \dR\times [0,\infty)$. It is easy to see that for a finite interval $I\subset[0,\infty)$, the point process $\{x: (x,t)\in\cN, t\in I\}$ is Poisson point process  with intensity $|I| e^{x}dx$, which has the same distribution as $\{x_k - \log|I|,k\ge1\}$ where $\{x_k:k\ge 1\}$ is a Poisson point process  with intensity $e^{x}dx$. Thus for an interval $I$ of length $\theta D_{\infty}$ independent of $\cN$, we have
\begin{align*}
\biggl(\sum_{k\ge1}e^{-\gb_i(x_k-\log D_\infty)}&\sum_{y\in V_k}e^{-\gb_i y}, i=1,2,\ldots,k\biggr)
\equald
\biggl( \sum_{(x,t)\in\cN: t\in I}e^{-\gb_ix}\sum_{y\in V_{(x,t)}}e^{-\gb_i y}, i=1,2,\ldots,k\biggr)
\end{align*}
Now fix an integer $k\ge1$ and consider the set of vertices $v\in\vT,|v|=k$ in the tree $\vT$ at level $k$. Consider the collection of random variables $(n^{3\gb/2}Z_n(\gb;v),|v|=k)$ which clearly are i.i.d.~and by the above reasoning has the limit (in distribution)
\[
\biggl( \sum_{(x,t)\in\cN: t/\theta\in I(v)}e^{-\gb x}\sum_{y\in V_{(x,t)}}e^{-\gb y}, |v|=k\biggr)
\]
where $I(v),|v|=k$ are mutually disjoint intervals of length $\theta D_{\infty}(v)$ and $\{D_{\infty}(v),|v|=k\}$ are i.i.d.~copies of $D_{\infty}$. From here the proof of Theorem~\ref{mainthm} follows easily.\qed

The following revealing calculations are also direct consequences.
\begin{cor}
Under conditions of the theorem,
\[
\lim_{\gb\to\infty}\gC(1-1/\gb)^{-1}n^{{\frac32}}Z_n(\gb)^{1/\gb} = n^{{\frac32}} e^{-\min_{|s|=n} H(s)}
\]
and
\[
\lim_{\gb\to\infty} \E(||\{e^{-y},y\in V\}||_\gb)    \left(T_{\theta D_\infty}^{(1/\gb)}\right)^{1/\gb} \equald \E(\max_{y\in V}y)\cdot D_{\infty}\cdot G
\]
where $-\log G$ has Gumble extreme value distribution. 
\end{cor}
\begin{proof}
A  consequence of Theorem~\ref{mainthm} is that the limiting distribution of $\gC(1-1/\gb)^{-\gb}  n^{3\gb/2}Z_{n}(\gb)$ is the same as a $\ga-$stable subordinator $T_{t}^{(\ga)}$ stopped at an independent r.v.~$\theta D_{\infty} \E(||\{e^{-y},y\in V\}||_\gb)$, where $\ga=1/\gb$, and $||g(y), y\in V||_\beta, \beta > 1,$ denotes
the usual $L^\beta-$norm, $\big(\int_{\dR}e^{-y\beta}V(dy)\big)^{\frac{1}{\beta}}$ with respect to the
decorating points.  
\end{proof}

As a closing remark one may view the ``genealogical structure'' of the resulting a.s.~defined
strong disorder cascade probability limit as follows:   If vertices are chosen from the $n$-th level according to the cascade measure in strong disorder, most of the branching occurs either within distance $o(n)$ from the root or within distance $o(n)$ from the $n$-th level. The branching near the $n$-th level gives rise to the decoration Point process in the limiting decorated Poisson process, whereas the Poisson process arises out of the time spent without any branching; see~\cites{derridaspohn, arguinbovierkistler} for 
comparison with branching Brownian motion.
Our result gives the structure near the root within distance $O(1)$, as discussed 
earlier. See Figure~\ref{fig1} for a graphical depiction.

\begin{figure}[htbp]
   \centering
   \includegraphics[width=6cm, height=3cm]{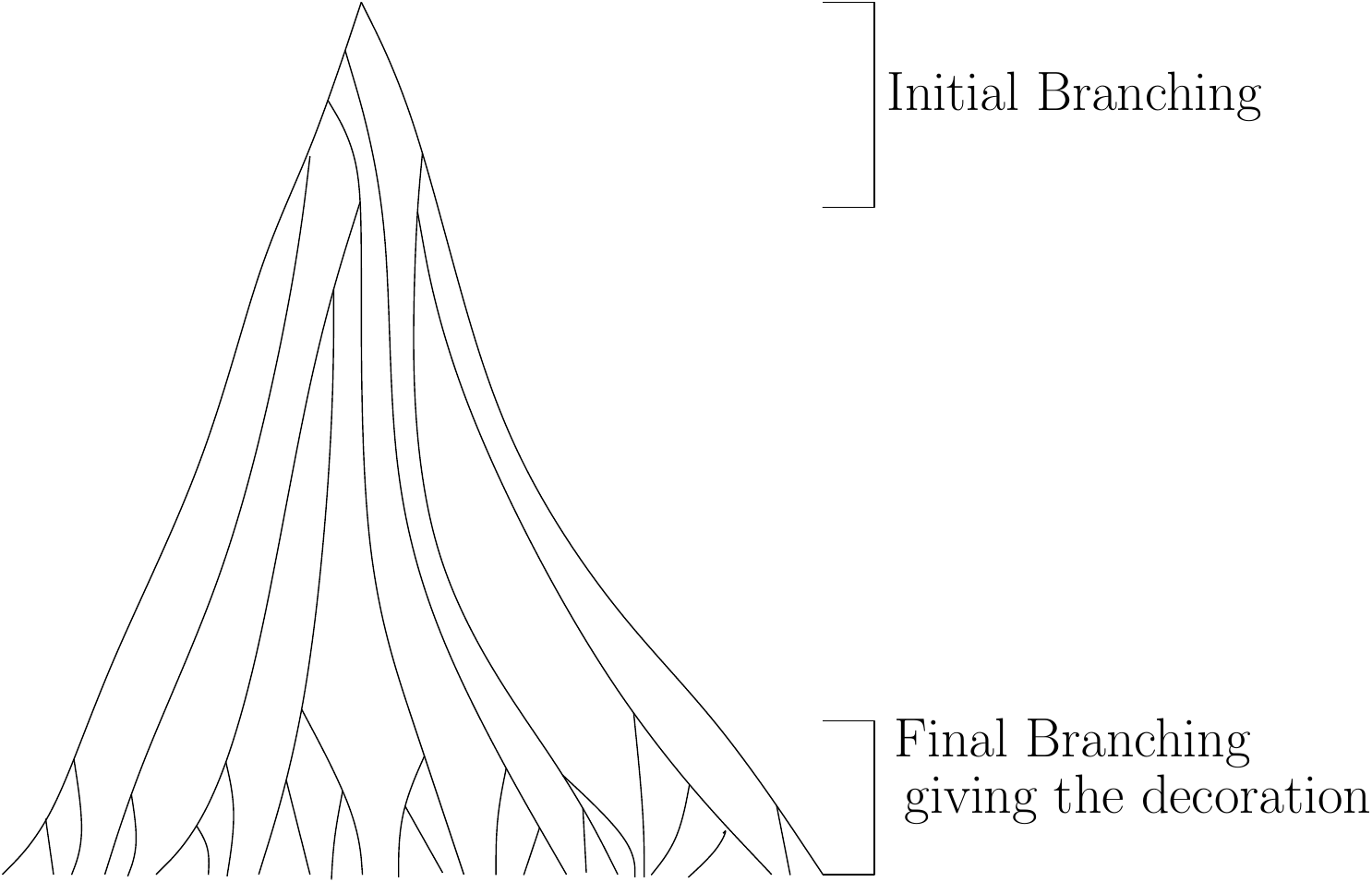} 
   \caption{Geneological structure in Branching Random Walk}
   \label{fig1}
\end{figure}

Another genealogical structure can be identified in terms of the L\`evy stable subordinator
$\{T_s^{(\ga)}: s\ge 0\}$ by viewing it as a continuous state branching process (csbp),
 in a manner as was done 
in  \cite{bertoinlegall00}  in describing the genealogy of Neveu's csbp
associated with another disordered system; namely, Derrida's  generalized random energy model (GREM). In particular it was shown in (\cite{bertoinlegall00}, Theorem 4) that 
the genealogy of Neveu's csbp defines a Bolthausen-Sznitman coalescent (BSC).
This could be accomplished by exploiting an alternative cascade version of GREM, due to Ruelle in 
\cite{ruelle87}.
Now observe that 
the (BSC) is a $\Lambda$-coalescent for a uniform distribution $\Lambda$
on $[0,1]$; see \cite{pitman}.  So,  in view of recent results of \cite{BBetal05}, the genealogy
of $\{T_s^{(\ga)}: s\ge 0\}$ is that of a $\Lambda-$coalescent for which
$\Lambda$ is a Beta distribution with parameters ${\beta_c/ \beta}$
and $1-{\beta_c/\beta}.$   Since ${\beta_c/ \beta} < 1$ under strict strong disorder, the results here establish interesting points of contrast  and comparison for these respective models of disorder; also see \cite{derridaspohn} for  other observations in this regard.

\vskip.1in
\noindent{\bf Acknowledgments.} This work began when PD was a Simons Postdoctoral fellow at the Courant Institute of Mathematical Sciences, New York University, where EW was a visitor. EW is also partially supported by a grant DMS-1408947 from the National Science Foundation.


\bibliographystyle{plain}
\bibliography{sorder}

\begin{bibdiv}
\begin{biblist}

\bib{aidekon}{article}{
      author={A{\"{\i}}d{\'e}kon, Elie},
       title={Convergence in law of the minimum of a branching random walk},
        date={2013},
        ISSN={0091-1798},
     journal={Ann. Probab.},
      volume={41},
      number={3A},
       pages={1362\ndash 1426},
         url={http://dx.doi.org/10.1214/12-AOP750},
      review={\MR{3098680}},
}

\bib{aidshi}{article}{
      author={Aidekon, Elie},
      author={Shi, Zhan},
       title={The {S}eneta-{H}eyde scaling for the branching random walk},
        date={2014},
        ISSN={0091-1798},
     journal={Ann. Probab.},
      volume={42},
      number={3},
       pages={959\ndash 993},
         url={http://dx.doi.org/10.1214/12-AOP809},
      review={\MR{3189063}},
}

\bib{aizwehr}{article}{
      author={Aizenman, Michael},
      author={Wehr, Jan},
       title={Rounding effects of quenched randomness on first-order phase
  transitions},
        date={1990},
        ISSN={0010-3616},
     journal={Comm. Math. Phys.},
      volume={130},
      number={3},
       pages={489\ndash 528},
  url={http://0-projecteuclid.org.pugwash.lib.warwick.ac.uk/getRecord?id=euclid.cmp/1104200601},
      review={\MR{1060388 (91e:82004)}},
}

\bib{aldous}{article}{
      author={Aldous, David~J.},
       title={The {$\zeta(2)$} limit in the random assignment problem},
        date={2001},
        ISSN={1042-9832},
     journal={Random Structures Algorithms},
      volume={18},
      number={4},
       pages={381\ndash 418},
         url={http://dx.doi.org/10.1002/rsa.1015},
      review={\MR{1839499 (2002f:60015)}},
}

\bib{arguinbovierkistler}{article}{
      author={Arguin, L.-P.},
      author={Bovier, A.},
      author={Kistler, N.},
       title={Genealogy of extremal particles of branching brownian motion},
        date={2011},
        ISSN={1097-0312},
     journal={Communications on Pure and Applied Mathematics},
      volume={64},
      number={12},
       pages={1647\ndash 1676},
         url={http://dx.doi.org/10.1002/cpa.20387},
}

\bib{abk13}{article}{
      author={Arguin, Louis-Pierre},
      author={Bovier, Anton},
      author={Kistler, Nicola},
       title={The extremal process of branching {B}rownian motion},
        date={2013},
        ISSN={0178-8051},
     journal={Probab. Theory Related Fields},
      volume={157},
      number={3-4},
       pages={535\ndash 574},
         url={http://dx.doi.org/10.1007/s00440-012-0464-x},
      review={\MR{3129797}},
}

\bib{barralrhodesvargas}{article}{
      author={Barral, J.},
      author={Rhodes, R.},
      author={Vargas, V.},
       title={Limiting laws of supercritical branching random walks},
        date={2012},
     journal={Comptes Rendus Mathematique},
      volume={350},
      number={9},
       pages={535\ndash 538},
}

\bib{berestyckietal}{article}{
      author={Berestycki, Julien},
      author={Berestycki, Nathana{\"e}l},
      author={Schweinsberg, Jason},
       title={The genealogy of branching {B}rownian motion with absorption},
        date={2013},
        ISSN={0091-1798},
     journal={Ann. Probab.},
      volume={41},
      number={2},
       pages={527\ndash 618},
         url={http://dx.doi.org/10.1214/11-AOP728},
      review={\MR{3077519}},
}

\bib{bertoinlegall00}{article}{
      author={Bertoin, Jean},
      author={Le~Gall, Jean-Fran{\c{c}}ois},
       title={The {B}olthausen-{S}znitman coalescent and the genealogy of
  continuous-state branching processes},
        date={2000},
        ISSN={0178-8051},
     journal={Probab. Theory Related Fields},
      volume={117},
      number={2},
       pages={249\ndash 266},
}

\bib{big}{article}{
      author={Biggins, J.~D.},
       title={The first- and last-birth problems for a multitype age-dependent
  branching process},
        date={1976},
        ISSN={0001-8678},
     journal={Advances in Appl. Probability},
      volume={8},
      number={3},
       pages={446\ndash 459},
      review={\MR{0420890 (54 \#8901)}},
}

\bib{bigkyp04}{article}{
      author={Biggins, J.~D.},
      author={Kyprianou, A.~E.},
       title={Measure change in multitype branching},
        date={2004},
        ISSN={0001-8678},
     journal={Adv. in Appl. Probab.},
      volume={36},
      number={2},
       pages={544\ndash 581},
         url={http://dx.doi.org/10.1239/aap/1086957585},
      review={\MR{2058149 (2005f:60179)}},
}

\bib{bigkyp05}{article}{
      author={Biggins, J.~D.},
      author={Kyprianou, A.~E.},
       title={Fixed points of the smoothing transform: the boundary case},
        date={2005},
        ISSN={1083-6489},
     journal={Electron. J. Probab.},
      volume={10},
       pages={no. 17, 609\ndash 631},
         url={http://dx.doi.org/10.1214/EJP.v10-255},
      review={\MR{2147319 (2006j:60086)}},
}

\bib{BBetal05}{article}{
      author={Birkner, Matthias},
      author={Blath, Jochen},
      author={Capaldo, Marcella},
      author={Etheridge, Alison},
      author={M{\"o}hle, Martin},
      author={Schweinsberg, Jason},
      author={Wakolbinger, Anton},
       title={Alpha-stable branching and beta-coalescents},
        date={2005},
        ISSN={1083-6489},
     journal={Electron. J. Probab.},
      volume={10},
       pages={no. 9, 303\ndash 325 (electronic)},
  url={http://0-dx.doi.org.pugwash.lib.warwick.ac.uk/10.1214/EJP.v10-241},
      review={\MR{2120246 (2006c:60100)}},
}

\bib{bolthausen91}{incollection}{
      author={Bolthausen, Erwin},
       title={On directed polymers in a random environment},
        date={1991},
   booktitle={Selected {P}roceedings of the {S}heffield {S}ymposium on
  {A}pplied {P}robability ({S}heffield, 1989)},
      series={IMS Lecture Notes Monogr. Ser.},
      volume={18},
   publisher={Inst. Math. Statist.},
     address={Hayward, CA},
       pages={41\ndash 47},
         url={http://dx.doi.org/10.1214/lnms/1215459286},
      review={\MR{1193060}},
}

\bib{bovier}{misc}{
      author={Bovier, Anton},
       title={Statistical mechanics of disordered systems. a mathematical
  perspective. cambridge series in statistical and probabilistic mathematics},
   publisher={Cambridge University Press, Cambridge},
        date={2006},
}

\bib{brunetderrida}{article}{
      author={Brunet, {\'E}ric},
      author={Derrida, Bernard},
       title={A branching random walk seen from the tip},
        date={2011},
     journal={Journal of Statistical Physics},
      volume={143},
      number={3},
       pages={420\ndash 446},
}

\bib{burtonway86}{article}{
      author={Burton, Robert~M., Jr.},
      author={Waymire, Ed},
       title={A sufficient condition for association of a renewal process},
        date={1986},
        ISSN={0091-1798},
     journal={Ann. Probab.},
      volume={14},
      number={4},
       pages={1272\ndash 1276},
      review={\MR{866348}},
}

\bib{derridaspohn}{article}{
      author={Derrida, B.},
      author={Spohn, H.},
       title={Polymers on disordered trees, spin glasses, and traveling waves},
        date={1988},
     journal={Journal of Statistical Physics},
      volume={51},
      number={5/6},
       pages={817\ndash 840},
}

\bib{derrida85}{article}{
      author={Derrida, Bernard},
       title={A generalization of the random energy model which includes
  correlations between energies},
        date={1985},
     journal={Journal de Physique Lettres},
      volume={46},
      number={9},
       pages={401\ndash 407},
}

\bib{durrettliggett83}{article}{
      author={Durrett, Richard},
      author={Liggett, Thomas~M.},
       title={Fixed points of the smoothing transformation},
        date={1983},
        ISSN={0044-3719},
     journal={Z. Wahrsch. Verw. Gebiete},
      volume={64},
      number={3},
       pages={275\ndash 301},
         url={http://dx.doi.org/10.1007/BF00532962},
      review={\MR{716487 (85e:60059)}},
}

\bib{evans}{article}{
      author={Evans, S.N.},
       title={Association and random measures},
        date={1989},
     journal={Probab. Theory. Rel. Fields},
      volume={86},
       pages={1\ndash 19},
}

\bib{holleyway}{article}{
      author={Holley, R.},
      author={Waymire, E.},
       title={Multifractal dimensions and scaling exponents for strangley
  bounded random cascades},
        date={1992},
     journal={Ann. Appld. Probab.},
      volume={2},
      number={4},
       pages={819\ndash 845},
}

\bib{hu_shi}{article}{
      author={Hu, Yueyun},
      author={Shi, Zhan},
       title={Minimal position and critical martingale convergence in branching
  random walks, and directed polymers on disordered trees},
        date={2009},
        ISSN={0091-1798},
     journal={Ann. Probab.},
      volume={37},
      number={2},
       pages={742\ndash 789},
         url={http://dx.doi.org/10.1214/08-AOP419},
      review={\MR{2510023}},
}

\bib{john_way}{article}{
      author={Johnson, Torrey},
      author={Waymire, Edward~C.},
       title={Tree polymers in the infinite volume limit at critical strong
  disorder},
        date={2011},
        ISSN={0021-9002},
     journal={J. Appl. Probab.},
      volume={48},
      number={3},
       pages={885\ndash 891},
}

\bib{kahanepeyriere}{article}{
      author={Kahane, J.-P.},
      author={Peyri{\`e}re, J.},
       title={Sur certaines martingales de {B}enoit {M}andelbrot},
        date={1976},
        ISSN={0001-8708},
     journal={Advances in Math.},
      volume={22},
      number={2},
       pages={131\ndash 145},
      review={\MR{0431355}},
}

\bib{kyprianou}{article}{
      author={Kyprianou, A.~E.},
       title={Slow variation and uniqueness of solutions to the functional
  equation in the branching random walk},
        date={1998},
        ISSN={0021-9002},
     journal={J. Appl. Probab.},
      volume={35},
      number={4},
       pages={795\ndash 801},
      review={\MR{1671230 (2000d:60138)}},
}

\bib{madaule}{article}{
      author={Madaule, Thomas},
       title={Convergence in law for the branching random walk seen from its
  tip},
        date={2011},
      eprint={http://arxiv.org/abs/1107.2543},
}

\bib{maillard13}{article}{
      author={Maillard, Pascal},
       title={A note on stable point processes occurring in branching brownian
  motion},
        date={2013},
     journal={Electronic Communications in Probability},
      volume={18},
      number={5},
       pages={1\ndash 9},
}

\bib{neveu92}{article}{
      author={Neveu, J.},
       title={A continuous-state branching process in relation with the grem
  model of spin glass theory},
        date={1992},
     journal={Rapport interne},
      volume={267},
}

\bib{newman_stein}{article}{
      author={Newman, C.M.},
      author={Stein, D.L.},
       title={Thermodynamic chaos and the structure of short-range spin
  glasses},
        date={1997},
     journal={Phys. Rev. E Part A},
      volume={3},
      number={5},
       pages={5194\ndash 5211},
}

\bib{pitman}{book}{
      author={Pitman, Jim},
       title={Combinatorial stochastic processes},
      series={Springer Lecture Notes in Mathematics},
   publisher={Springer-Verlag},
        date={2002},
        note={Lectures from the 32nd Summer School on Probability Theory held
  in Saint-Flour, 2002},
}

\bib{ruelle87}{article}{
      author={Ruelle, David},
       title={A mathematical reformulation of {D}errida's {REM} and {GREM}},
        date={1987},
        ISSN={0010-3616},
     journal={Comm. Math. Phys.},
      volume={108},
      number={2},
       pages={225\ndash 239},
}

\end{biblist}
\end{bibdiv}

%
%
\end{document}